\newif\ifsocg
\socgfalse
\ifsocg
\documentclass[a4paper,UKenglish]{lipics-v2016}
\usepackage{microtype}
\usepackage{multirow}

\else
\documentclass[a4paper,10pt]{article}
\usepackage{fullpage}
\usepackage{authblk}
\fi

\usepackage[utf8]{inputenc}

\usepackage{amsfonts, amssymb, amsmath, amsthm}
\usepackage{mathtools}
\usepackage{enumerate}
\usepackage{graphicx}
\usepackage[usenames,dvipsnames,svgnames,table]{xcolor}

\usepackage{hyperref}
\hypersetup{colorlinks=true, linkcolor=red, urlcolor=blue, citecolor=blue, pdfpagemode=UseNone, pdfstartview=FitH, bookmarksopen=true}
\hypersetup{pdftitle=Title}

\ifsocg
\else
\newtheorem{theorem}{Theorem}
\newtheorem{lemma}[theorem]{Lemma}
\newtheorem{corollary}[theorem]{Corollary}
\theoremstyle{remark}
\newtheorem{remark}[theorem]{Remark}
\theoremstyle{definition}

\fi

\theoremstyle{theorem}
\newtheorem*{theorem*}{Theorem}
\newtheorem*{lemma*}{Lemma}
\newtheorem*{proposition*}{Proposition}
\newtheorem{proposition}[theorem]{Proposition}
\newtheorem*{fact*}{Fact}

\newtheorem*{question*}{Question}

\newtheorem*{corollary*}{Corollary}

\numberwithin{claimcounter}{theorem}
\newtheorem*{claim*}{Claim}

\theoremstyle{remark}
\newtheorem*{remark*}{Remark}

\theoremstyle{definition}
\newtheorem*{definition*}{Definition}

\numberwithin{equation}{section}

\newtheorem*{observation*}{Observation}

\ifsocg
\newcommand{\heading}[1]{\subparagraph{#1}}
\else
\newcommand{\heading}[1]{\paragraph{#1}}
\fi

\newcommand{\R}{\mathbb{R}}
\newcommand{\Z}{\mathbb{Z}}

\newcommand*{\medcup}{\mathop{\scalebox{1}{\ensuremath{\bigcup}}}\nolimits}

\DeclareMathOperator{\sd}{sd}
\DeclareMathOperator{\lk}{lk}
\DeclareMathOperator{\conv}{conv}
\newcommand{\wh}[1]{\widehat{#1}}

\DeclareMathOperator{\andd}{and}
\newcommand{\vand}{v_{\andd}}
\newcommand{\fand}{f_{\andd}}

\newcommand{\true}{\textsc{true}}
\newcommand{\false}{\textsc{false}}

\newcommand{\gand}{{\bf A}}

\newcommand{\gadg}{{\bf G}}

\newcommand{\cla}{{\bf C}}
\newcommand{\var}{{\bf V}}

\newcommand{\pth}[1]{\left( #1 \right)}

\newif\ifcmts
\cmtstrue   

\ifcmts
\newcommand{\marrow}{\marginpar{\boldmath$\longleftarrow$}}
\newcommand{\martin}[1]{\ifhmode\newline\fi\marrow
  \textsf{\textcolor{green}{\bf
MARTIN:} #1\newline}}
\newcommand{\uli}[1]{\ifhmode\newline\fi\marrow
  \textsf{\textcolor{red}{\bf
ULI:} #1\newline}}
\newcommand{\xavier}[1]{\ifhmode\newline\fi\marrow
  \textsf{\textcolor{blue}{\bf
Xavier:} #1\newline}}
\newcommand{\zuzka}[1]{\ifhmode\newline\fi\marrow
  \textsf{\textcolor{magenta}{\bf
Zuzka:} #1\newline}}
\newcommand{\pavel}[1]{\ifhmode\newline\fi\marrow
  \textsf{\textcolor{teal}{\bf
Pavel:} #1\newline}}
\else
\newcommand{\martin}[1]{}
\newcommand{\uli}[1]{}
\newcommand{\xavier}[1]{}
\newcommand{\zuzka}[1]{}
\newcommand{\pavel}[1]{}
\fi

\newif\ifMpic
\Mpicfalse 

\title{Shellability is \textup{NP}-complete\footnote{
Partially supported by the Czech-French collaboration project EMBEDS
II (CZ: 7AMB17FR029, FR: 38087RM). XG is partially supported by IUF. 
MT is partially supported by the GA\v{C}R grant 16-01602Y.}}

\author[1]{Xavier Goaoc}
\author[2]{Pavel Pat\'ak}
\author[3]{Zuzana Pat\'akov\'a}
\author[4]{Martin Tancer}
\author[3]{Uli Wagner}

\affil[1]{\small Universit\'e Paris-Est, LIGM (UMR 8049), CNRS, ENPC, ESIEE, UPEM, F-77454, Marne-la-Vall\'ee, France.}
\affil[2]{\small Department of Mathematics and Statistics, Masaryk University, Brno, Czech Republic.}
\affil[3]{IST Austria, Klosterneuburg, Austria.}
\affil[4]{\small Department of Applied Mathematics, Charles University, Prague, Czech Republic.}

\ifsocg
\Copyright{Xavier Goaoc, Pavel Pat\'ak, Zuzana Pat\'akov\'a, Martin Tancer and
Uli Wagner}
\fi

\begin{document}

\maketitle

\begin{abstract}
  We prove that for every $d\geq 2$, deciding if a pure, $d$-dimensional, simplicial
  complex is shellable is \textup{NP}-hard, hence \textup{NP}-complete. 
  This resolves a question raised, e.g., by Danaraj and Klee in 1978. Our reduction
  also yields that for every $d \ge 2$ and $k \ge 0$, deciding if a
  pure, $d$-dimensional, simplicial complex is $k$-decomposable is
  \textup{NP}-hard. For $d \ge 3$, both problems remain \textup{NP}-hard when restricted
  to contractible pure $d$-dimensional complexes. 
  \ifsocg
  \else
Another simple corollary of our result is that it is NP-hard to decide whether a given poset is CL-shellable.
  \fi
\end{abstract}

\section{Introduction}
A $d$-dimensional simplicial complex is called \emph{pure} if all its
facets (i.e., inclusion-maximal faces) have the same dimension $d$. 
A pure $d$-dimensional simplicial complex is \emph{shellable} if there exists
a linear ordering $\sigma_1, \sigma_2, \ldots, \sigma_n$ of its facets 
such that,  for every $i \ge 2$, $\sigma_i \cap \pth{\cup_{j < i} \sigma_j}$ is a pure
$(d-1)$-dimensional simplicial complex; such an ordering is called a \emph{shelling} or \emph{shelling order}. 
\ifsocg \else

\fi
For example, the boundary of a simplex is shellable (any order works), 
but no triangulation of the torus is  
(the condition fails for the first triangle $\sigma_i$ that
creates a non-contractible $1$-cycle).

The concept of shellings originated in the theory of convex polytopes 
(in a more general version for polytopal complexes
), 
as an inductive procedure  to construct the boundary of a polytope by adding the facets 
one by one in such a way that all intermediate complexes (except the last one) are contractible.
The fact that this is always possible, i.e., that 
convex polytopes are shellable, was
initially used as an unproven assumption in early papers (see the discussion in \cite[pp.~141--142]{Grunbaum:Convex-polytopes-2003} for a more detailed account of the history), before being proved by Bruggesser and Mani~\cite{bruggesser1972shellable}.

The notion of shellability extends to more general objects (including
non-pure simplicial complexes and posets~\cite{bjorner97}),
and plays an important role in diverse areas including piecewise-linear 
topology\cite{Rourke:Introduction-to-piecewise-linear-topology-1982,Bing:The-geometric-topology-of-3-manifolds-1983}, 
polytope theory (e.g., McMullen's proof of the \emph{Upper Bound Theorem}~\cite{mcmullen1970maximum}%
), topological combinatorics \cite{Bjorner:Topological-methods-1995}, algebraic
combinatorics and commutative algebra
\cite{Stanley:Combinatorics-and-commutative-algebra-1996,
peeva-reiner-sturmfels98}
, poset theory
, and group theory 
\cite{Bjorner:Shellable-and-Cohen-Macaulay-partially-ordered-1980,Shareshian:On-the-shellability-of-the-order-complex-of-the-subgroup-2001}; for a more detailed
introduction 
and further references see 
\cite[$\mathsection 3$]{wachs2007poset}.

One of the reasons for its importance is that shellability---a combinatorial
property---has strong topological implications: For example, if a pure
$d$-dimensional complex $K$ is a \emph{pseudomanifold}\footnote{A pure,
$d$-dimensional complex $K$ is a pseudomanifold (with boundary) if every
$(d-1)$-face of $K$ is contained in exactly two (at most two) facets.
(Sometimes, it is additionally required that the facet-adjacency graph of $K$
is connected, but this does not matter in our setting, since shellable
complexes always satisfy this connectivity property.)}---which can be checked
in linear time---and shellable, then $K$ is homeomorphic to the sphere $S^d$
(or the ball $B^d$, in case $K$ has nonempty boundary)
\cite{danaraj1974shellings}---a property that is algorithmically undecidable
for $d\geq 5$, by a celebrated result of Novikov
\cite{Volodin:The-problem-of-discriminating-algorithmically-the-standard-1974,Nabutovsky:Einstein-structures:-Existence-versus-1995}.
More generally, every pure $d$-dimensional shellable complex is homotopy
equivalent to a wedge of $d$-spheres, in particular it is $(d-1)$-connected.

\subsection{Results}
From a computational viewpoint, it is natural to ask if one can decide efficiently (in polynomial time) whether a given complex is shellable. This question was raised at least as early as in the 1970's \cite{Danaraj:A-representation-of-2-dimensional-pseudomanifolds-and-its-use-in-the-design-1978,Danaraj:Which-spheres-are-shellable-1978} (see also  \cite[Problem~34]{Kaibel:Some-algorithmic-problems-in-polytope-2003}) and is of both practical and theoretical importance (besides direct consequences for the experimental exploration of simplicial complexes, the answer is also closely related to the question there are simple conditions that would characterize shellability).
Danaraj and Klee~proved that shellability of $2$-dimensional pseudomanifolds can be tested in linear time \cite{Danaraj:A-representation-of-2-dimensional-pseudomanifolds-and-its-use-in-the-design-1978}, whereas a number of related problems have been shown to be \textup{NP}-complete 
\cite{Egecioglu:A-computationally-intractable-problem-on-simplicial-1996,Lewiner:Optimal-discrete-Morse-functions-2003,joswig-pfetsch06,malgouyres-frances08,tancer16} (see Section~\ref{s:collapsibility}), but the computational complexity of the shellability problem has remained open.
\ifsocg
\else 
Here, we settle the question in the negative and show that the problem is
intractable (modulo $\textup{P}\neq\textup{NP}$).%
\fi%
\footnote{For basic notions from computational complexity, such as \textup{NP}-completeness or reductions, see, e.g., \cite{Arora:Complexity-Theory:-A-Modern-Approach-2009}.}

\begin{theorem}
\label{t:main}
Deciding if a pure $2$-dimensional simplicial complex is shellable is \textup{NP}-complete.
\end{theorem}
Here, the input is given as a finite abstract simplicial complex (see Section~\ref{s:back}).%
\footnote{There are a several different ways of encoding an abstract simplicial complex---e.g., we can list the facets, or we can list all of its simplices---, but since we work with complexes of fixed dimension, these encodings can be translated into one another in polynomial time, so the precise choice does not matter.} 
\begin{remark}
The problem of testing shellability is easily seen to lie in the complexity class \textup{NP} (given a linear ordering of the facets of a complex,
it is straightforward to check whether it is a shelling). Thus, the nontrivial part of Theorem~\ref{t:main} is that 
deciding shellability of pure $2$-dimensional complexes is \textup{NP}-hard.
\end{remark}

It is easy to check that a pure simplicial complex $K$ is shellable if and only if the \emph{cone} $\{v\} * K$ is
shellable, where $v$ is a vertex not in $K$ (see Section~\ref{s:back}). Thus, the hardness of deciding shellability 
easily propagates to higher-dimensional complexes, even to cones.

\begin{corollary}\label{cor:Shell-Contractible}
For $d \ge 3$, deciding if a pure $d$-dimensional complex is shellable is \textup{NP}-complete
even when the input is assumed to be a cone (hence contractible).
\end{corollary}

Moreover, our hardness reduction (from $3$-SAT) used in the proof of Theorem~\ref{t:main} (see Section~\ref{s:results})
turns out to be sufficiently robust to also imply hardness results for a number
of related problems.

\ifsocg
\else
\heading{Hardness of $k$-decomposability.}
\fi
Let $d \ge 2$ and $k \ge 0$. A pure $d$-dimensional simplicial complex
$K$ is \emph{$k$-decomposable} if it is a simplex or if there exists a face
$\sigma$ of $K$ of dimension at most $k$ such that (i) the link of
$\sigma$ in $K$ is pure $(d-|\sigma|)$-dimensional and
$k$-decomposable, and (ii) deleting $\sigma$ and faces of $K$ containing $\sigma$
produces a $d$-dimensional $k$-decomposable complex. 
This notion, introduced by Provan and Billera~\cite{billera80}, provides
a hierarchy of properties ($k$-decomposability implies $(k+1)$-decomposability) 
interpolating between \emph{vertex-decomposable} complexes ($k=0$) 
and shellable complexes (shellability is equivalent to $d$-decomposability \cite{billera80}).
The initial motivation for considering this hierarchy was to study the \emph{Hirsch conjecture}
on combinatiorial diameters of convex polyhedra, or in the language of simplicial complex,
the diameter of the facet-adjacency graphs of pure simplicial complexes: at one end, the 
boundary complex of every $d$-dimensional simplicial polytope is shellable~\cite{bruggesser1972shellable}, 
and at the other end, every $0$-decomposable simplicial complex has small diameter 
(it satisfies the \emph{Hirsch bound}~\cite{billera80}).

\begin{theorem}
\label{t:decomposability}
  Let $d \ge 2$ and $k \ge 0$. Deciding if a pure $d$-dimensional
  simplicial complex is $k$-decomposable is \textup{NP}-hard. For $d \ge 3$,
  the problem is already \textup{NP}-hard for pure $d$-dimensional
 simplicial complexes that are cones (hence contractible). 
\end{theorem}

\ifsocg\else
\heading{Hardness of CL-shellability of posets.}
\fi
Another notion related to shellability is the \emph{CL-shellability}
of a poset, introduced in~\cite{bjorner82}. 
\ifsocg\else The definition of CL-shellability is
rather technical, so we do not reproduce it \fi here, but \ifsocg
\else simply
\fi
note that a simplicial complex is shellable if and only if the dual of
its face lattice is CL-shellable~\cite[Corollary
  4.4]{bjorner83}. \ifsocg
\else Since for any fixed dimension $d$, the face lattice has height $d+2$ and
  can be computed in time polynomial in the size of the $d$-complex we get:\fi

\begin{corollary}%
  \ifsocg
\else%
  For any fixed $d\geq 4$, deciding CL-shellability of posets of height at most $d$ is \textup{NP}-hard.
  \fi%
\end{corollary}

\subsection{Related Work on Collapsibility and Our Approach}
\label{s:collapsibility}

Our proof of Theorem~\ref{t:main} builds on earlier results concerning \emph{collapsibility}, a combinatorial 
analogue, introduced by Whitehead~\cite{Whitehead:Simplicial-Spaces-Nuclei-and-m-Groups-1939}, 
of the topological notion of contractibility.\footnote{%
Collapsibility implies contractibility, but the latter property is undecidable for complexes of dimension at least $4$
(this follows from Novikov's result~\cite{Volodin:The-problem-of-discriminating-algorithmically-the-standard-1974},
see \cite[Appendix~A]{tancer16}), whereas the problem of deciding collapsibility lies in NP.} 
A face $\sigma$ of a simplicial complex $K$ is \emph{free} if there is a unique inclusion-maximal face $\tau$ 
of $K$ with $\sigma \subsetneq \tau$. An \emph{elementary collapse} is the operation of deleting a free face and all faces
containing it. A simplicial complex $K$ \emph{collapses} to a subcomplex $L \subseteq K$ if $L$ can be obtained from $K$ 
by a finite sequence of elementary collapses; $K$ is called \emph{collapsible} if it collapses to a single vertex. 

The problem of deciding whether a given $3$-dimensional complex is collapsible is NP-complete~\cite{tancer16};
the proof builds on earlier work of Malgouyres and Franc\'{e}s~\cite{malgouyres-frances08}, who showed 
that it is NP-complete to decide whether a given $3$-dimensional complex collapses to some $1$-dimensional 
subcomplex. By contrast, collapsibility of $2$-dimensional complexes can be decided in polynomial time (by a greedy algorithm)~\cite{joswig-pfetsch06,malgouyres-frances08}. 
It follows that for any \emph{fixed} integer $k$, it can be decided in
polynomial time whether a given $2$-dimensional simplicial complex can be made
collapsible by deleting at most $k$ faces of dimension $2$; by contrast, the
latter problem is NP-complete if $k$ is part of the input
\cite{Egecioglu:A-computationally-intractable-problem-on-simplicial-1996}.\footnote{We
remark that building
on~\cite{Egecioglu:A-computationally-intractable-problem-on-simplicial-1996},
a related problem, namely computing \emph{optimal discrete Morse matchings} in
simplicial complexes (which we will not define here), was also shown to be
NP-complete~\cite{Lewiner:Optimal-discrete-Morse-functions-2003,joswig-pfetsch06}.}

Our reduction uses the gadgets introduced by Malgouyres and
Franc\'{e}s~\cite{malgouyres-frances08} and reworked in~\cite{tancer16}
to prove NP-hardness of deciding collapsibility for $3$-dimensional complexes.
However, these gadgets are not pure: they contain maximal simplices of two different
dimensions, $2$ and~$3$. Roughly speaking, we fix this by replacing the $3$-dimensional 
subcomplexes by suitably triangulated $2$-spheres and modifying the way in which they
are glued. Interestingly, this also makes our reduction robust to
subdivision and applicable to other types of decomposition.

\heading{Collapsibility and shellability.} Furthermore, we will use the following connection between shellability and collapsibility, 
due to Hachimori~\cite{hachimori08} (throughout, $\tilde \chi$ denotes the reduced Euler
characteristic).

\begin{theorem}[{\cite[Theorem~8]{hachimori08}}]
  \label{t:hachi}
  Let $K$ be a $2$-dimensional simplicial complex. The second
  barycentric subdivision $\sd^2 K$ is shellable if and only if the
  link of each vertex of $K$ is connected and there exists $\tilde
  \chi (K)$ triangles in $K$ whose removal makes $K$ collapsible.
\end{theorem}

At first glance, Hachimori's theorem might suggest to prove Theorem~\ref{t:main}
by a direct polynomial-time reduction of collapsibility to shellability. However, for
$2$-dimensional complexes this would not imply hardness, 
since, as mentioned above, collapsibility of $2$-dimensional complexes is
decidable in polynomial time~\cite{joswig-pfetsch06,malgouyres-frances08}.
Instead, we will use the existential part of Hachimori's theorem (``there
exists $\tilde \chi (K)$ triangles'') to encode instances of the 3-SAT problem, 
a classical NP-complete problem.

\ifsocg
\else
\section{Background and Terminology}
\fi
\label{s:back}

\ifsocg
\else
We include here a brief summary of the main notions that we use
(except for the notions already defined in the introduction, such as
pure, shellable, and collapsible simplicial complexes and free faces and 
elementary collapses).

\heading{Simplicial complexes.} A (finite) \emph{abstract simplicial complex}
is a collection $K$ of subsets of a finite set $V$ that is closed under under taking subsets, 
i.e., if  $\sigma \in K$ and $\tau \subseteq \sigma$, then $\tau \in K$. The elements $v\in V$
 are called the \emph{vertices} of 
$K$ (and often identified with the singleton sets $\{v\}\in K$), and the elements of $K$ are called \emph{faces} or \emph{simplices} of $K$.
The \emph{dimension} of a face is its cardinality minus $1$, and the
\emph{dimension} of $K$ is the maximum 
dimension of any face. This is a purely combinatorial description of a simplicial complex and a natural input model for computational questions.

For the purposes of exposition, in particular for describing the gadgets used
in the reduction, it will be more convenient to use an alternative, geometric
description of simplicial complexes: A (finite) \emph{geometric simplicial
complex} is a finite collection $K$ of geometric simplices (convex hulls of
affinely independent points) in $\R^d$ (for some $d$) such that (i) if $\sigma
\in K$ and $\tau$ is a face of $\sigma$, then $\tau$ also belongs to $K$, and
(ii) if $\sigma_1, \sigma_2 \in K$, then $\sigma_1 \cap \sigma_2$ is a face of
both $\sigma_1$ and $\sigma_2$. The \emph{polyhedron} of a geometric
simplicial complex $K$
is defined as the union of simplices contained in $K$,
$\bigcup_{\sigma \in K} \sigma$. We also say that $K$ \emph{triangulates} $X
\subseteq \R^d$ if $X$ is the polyhedron of $K$. 
Note that a given polyhedron usually has many different triangulations.

There is a straightforward way of translating between the two descriptions
(see, e.g.  \cite[Chapter~1]{Matousek:BorsukUlam-2003}):
On the one hand, for every geometric simplicial complex $K$ there is an 
associated abstract simplicial complex, namely the collection of sets of vertices of the simplices of $K$ 
(considered as finite sets, neglecting their geometric position). 
Conversely, for any given abstract simplicial complex $K$, there is
a geometric simplicial complex whose associated abstract simplicial complex
is (isomorphic to) $K$: For a sufficiently large $d$, 
pick affinely independent points $p_v\in \R^d$, one for each vertex $v$ of $K$,
and let the simplices of the geometric complex be the convex hulls of the point
sets $\{p_v\colon v\in \sigma\}$, for all $\sigma \in K$.

For the rest of the article we work in the setting of geometric simplicial
complexes (except for the definition of joins, see below, which is simpler for
abstract simplicial complexes),%
with the understanding that a geometric simplicial complex is simply
a convenient way to describe the associated abstract simplicial complex. (In
particular, we will 
not care about issues such as coordinate complexity of the geometric complex.)
\fi

\ifsocg
\else
\heading{Links, subdivisions, and joins.}
Let $K$ be a (geometric) simplical complex. The \emph{link} of a vertex $v$ in $K$ is
defined as 
$$\lk_K v := \{\sigma \in K\colon v \not\in \sigma \hbox{ and }
\conv(\{v\} \cup \sigma) \in K\}.$$

A \emph{subdivision} of a complex $K$ is a complex $K'$ such that 
the polyhedron of $K$ coincides with the polyhedron of $K'$
and every simplex of $K'$ is
contained in some simplex of $K$. 

We will use a specific class of subdivision, called barycentric
subdivision. Given a nonempty simplex $\sigma \in \R^d$, let
$b_\sigma$ denote its barycenter (we have $v = b_{v}$ for a vertex
$v$). The \emph{barycentric subdivsion} of a complex $K$ in $\R^d$ is
the complex $\sd K$ with vertex set $\{b_{\sigma} \colon \sigma \in K
\setminus \{\emptyset\}\}$ and whose simplices are of the form $\conv
(b_{\sigma_1}, \cdots, b_{\sigma_k})$ where $\sigma_1 \subsetneq
\sigma_2 \subsetneq \ldots \subsetneq \sigma_k \in K \setminus
\{\emptyset\}$. The \emph{$\ell$th barycentric subdivision} of $K$,
denoted $\sd^\ell K$, is the complex obtained by taking successively
the barycentric subdivision $\ell$ times.

The \emph{join} $K *L$ of two (abstract) simplicial complexes with disjoint sets
of vertices is the complex $K*L = \{\sigma \cup \tau \colon \sigma \in
K, \tau \in L\}$. In particular, note that $\{\emptyset\} * K =
K$. For $\ell\geq 0$, let $\Delta_\ell$ denote the \emph{$\ell$-dimensional simplex};%
\footnote{Considered as a simplicial complex consisting of all faces of the simplex, including the simplex itself; 
as an abstract simplicial complex, $\Delta_\ell$ consists of all the subsets of an $(\ell+1)$-element set.} 
we extend this to the case $\ell=-1$, by using the convention that $\Delta_{-1}=\{\emptyset\}$ 
is the abstract simplicial complex whose unique face is the empty face $\emptyset$. Note that, for a pure simplicial complex $K$, an ordering $\sigma_1, \sigma_2, \ldots, \sigma_n$ of the
facets of $K$ is shelling if and only if the ordering $\Delta_\ell \cup \sigma_1, \Delta_\ell \cup \sigma_2, \ldots,
\Delta_\ell \cup \sigma_n$ is a shelling of $\Delta_\ell * K$.

\heading{Reduced Euler characteristic.}
The \emph{reduced Euler characteristic} of a complex $K$
is defined as
$$
\tilde \chi (K) = \sum\limits_{i={-1}}^{\dim K} (-1)^i f_i(K)
$$
where $f_i(K)$ is the number of $i$-dimensional faces of $K$ and,
by convention, $f_{-1}(K)$ is $0$ if $K$ is empty and $1$ otherwise.
\fi

\ifsocg\else
  \heading{$3$-SAT problem.} For our reduction, we use the $3$-SAT problem
  (a classical \textup{NP}-hard problem). The $3$-SAT problem takes as
  input a $3$-CNF formula $\phi$, that is, a Boolean formula which is
  a conjunction of simpler formulas called \emph{clauses}; each clause
  is a disjunction of three literals, where a \emph{literal} is a
  variable or the negation of a variable. An example of $3$-CNF
  formula is
  \[ \phi' = (x_1 \vee x_2 \vee \neg x_3) \wedge
  (\neg x_1 \vee \neg x_2 \vee x_4).\]
  The \emph{size} of a formula is the total number of literals
  appearing in its clauses (counting repetitions). The output of the
  $3$-SAT problem states whether $\phi$ is satisfiable, that is,
  whether we can assign variables \true{} or \false{} so that the formula
  evaluates as \true{}.  The formula $\phi'$ given above is satisfiable, for example
  by setting $x_1$ to \true{}, $x_2$ to \false{} and $x_3$ and $x_4$
  arbitrarily. 
\fi

\section{The Main Proposition and its Consequences}
\label{s:results}

The cornerstone of our argument is the following construction:

\begin{proposition}
  \label{p:Kphi}
  There is an algorithm that, given a $3$-CNF formula%
   $\phi$, produces,
  in time polynomial in the size of $\phi$, a $2$-dimensional 
  simplicial complex $K_\phi$ with the following properties:
  \begin{enumerate}[(i)]
  \item the link of every vertex of $K_\phi$ is connected,
  \item if $\phi$ is satisfiable, then $K_\phi$ becomes collapsible after
    removing some $\tilde \chi (K_\phi)$ triangles,
  \item if an arbitrary subdivision of $K_\phi$ becomes collapsible
    after removing some $\tilde \chi (K_\phi)$ triangles, then $\phi$
    is satisfiable.
  \end{enumerate}
\end{proposition}

\noindent
The rest of this section derives our main result and its variants from
Proposition~\ref{p:Kphi}. We then describe the construction of
$K_\phi$ in Section~\ref{s:construction} and prove
Proposition~\ref{p:Kphi} in \ifsocg
\else Sections~\ref{s:links}
to~\ref{s:reverse}.\fi

\heading{Hardness of shellability.}
Proposition~\ref{p:Kphi} and Hachimori's theorem \ifsocg\else readily imply
\fi our
main result:

\begin{proof}[Proof of Theorem~\ref{t:main}]
  Let $\phi$ be a $3$-CNF formula and let $K_\phi$ denote the
  $2$-dimensional complex built according to
  Proposition~\ref{p:Kphi}. Since the link of every vertex of $K_\phi$
  is connected, Theorem~\ref{t:hachi} guarantees that $\sd^2 K_\phi$
  is shellable if and only if there exist $\tilde \chi (K_\phi)$
  triangles whose removal makes $K_\phi$ collapsible. Hence, by
  statements~(ii) and~(iii), the formula $\phi$ is satisfiable if and
  only if $\sd^2 K_\phi$ is shellable. Taking the barycentric
  subdivision of a two-dimensional complex multiplies its number of
  simplices by at most a constant factor. The complex $\sd^2 K_\phi$
  can thus be constructed from $\phi$ in polynomial time, and $3$-SAT
  reduces in polynomial time to deciding the shellability of
  $2$-dimensional pure complexes.
\end{proof}

\heading{Hardness of $k$-decomposability.}
Note that statement~(iii) in Proposition~\ref{p:Kphi} deals with
arbitrary subdivisions whereas Theorem~\ref{t:hachi} only mentions the
second barycentric subdivision. This extra elbow room comes at no cost
in our proof, and yields the \textup{NP}-hardness of $k$-decomposability.

\begin{proof}[Proof of Theorem~\ref{t:decomposability}]
  Assume without loss of generality that $k \le d$. Let $\phi$ be a
  $3$-CNF formula and $K_\phi$ the complex produced by
  Proposition~\ref{p:Kphi}. We have the following
  implications:\footnote{In the case $d=2$, we use the convention that $\Delta_{-1}\ast L=L$ for any simplicial complex $L$.}
\ifsocg
\else
  \[\begin{array}{rll}
   \phi \text{ is satisfiable} & \Rightarrow &  K_\phi \text{ is collapsible after removal of some $\tilde \chi(K_{\phi})$ triangles}\\
   & \Rightarrow & \sd^2 K_\phi \text{ is shellable }\\
   & \Rightarrow_{(b)} & \sd^3 K_\phi \text{ is vertex-decomposable}\\
   &  \Rightarrow_{(c)} & \Delta_{d-3} * \sd^3 K_\phi \text{ is vertex-decomposable}\\
   &    \Rightarrow_{(a)} & \Delta_{d-3} * \sd^3 K_\phi \text{ is $k$-decomposable}\\
  &  \Rightarrow_{(a)} & \Delta_{d-3} * \sd^3 K_\phi \text{ is shellable}\\
    &  \Rightarrow_{(d)} & \sd^3 K_\phi \text{ is shellable} \\
    &  \Rightarrow & \sd  K_\phi \text{ is collapsible after removal of some $\tilde \chi(K_{\phi})$ triangles}   \\
  & \Rightarrow & \phi \text{ is satisfiable}\\
  \end{array}\]
\fi
  The first and last implications are by construction of $K_\phi$
  (Proposition~\ref{p:Kphi}). The second and second to last follow
  from Theorem~\ref{t:hachi}, given that Proposition~\ref{p:Kphi}
  ensures that links of 
  vertices in $K_\phi$ are connected. The
  remaining implications follow from the following known \ifsocg\else facts (where
  $\Rightarrow_{(x)}$ to mean that the implication follows from
  observation~(x)):\fi

  \begin{enumerate}[(a)]
  \item if $K$ is $k$-decomposable, then $K$ is $k'$-decomposable for
    $k'\geq k$,
  \item if $K$ is shellable, then $\sd K$ is
    vertex-decomposable~\cite{bjorner97},
  \item $K$ is vertex-decomposable if and only if $\Delta_\ell * K$ is
    vertex decomposable~\cite[Prop.~2.4]{billera80},
  \item $K$ is shellable if and only if $\Delta_\ell * K$ is shellable
    \ifsocg\else (\emph{c.f.} Section~\ref{s:back}).\fi
  \end{enumerate}

\noindent  Since the first and last statement are identical, these are all
  equivalences. In particular, $\phi$ is satisfiable if and only if
  $\Delta_{d-3} * \sd^3 K_\phi$ is $k$-decomposable. Since this
  complex can be computed in time polynomial in the size of $K_\phi$,
  \ifsocg\else which is \fi polynomial in the size of $\phi$, the first statement
  follows. \ifsocg
\else For $d \ge 3$, $\Delta_{d-3} * \sd^3 K_\phi$ is
  contractible so the second statement follows.\fi
\end{proof}

\section{Construction}
\label{s:construction}

We now define the complex $K_\phi$ mentioned in
Proposition~\ref{p:Kphi}. This complex consists of several building
blocks, called \emph{gadgets}. We first give a ``functional'' outline
of the gadgets (in Section~\ref{s:outline}), insisting on the
properties that guided their design, before moving on to the details
of their construction and gluing (Sections~\ref{s:Bing}
and~\ref{s:detailed}).

We use the notational convention that complexes that depend on a
variable $u$ are denoted with round brackets, \emph{e.g.} $f(u)$,
whereas complexes that depend on a literal are denoted with square
brackets, \emph{e.g.} $f[u]$ or $f[\neg u]$.

\subsection{Outline of the construction}
\label{s:outline}
The gadgets forming $K_\phi$ are designed with two ideas in
mind. First, every gadget, when considered separately, can only be
collapsed starting in a few \emph{special edges}. Next, the special
edges of each gadgets are intended to be glued to other gadgets, so as
to create dependencies in the flow of collapses: if an edge~$f$ of a
gadget~$\gadg$ is attached to a triangle of another gadget
$\gadg'$, then~$\gadg$ cannot be collapsed starting by $f$ before some
part of $\gadg'$ has been collapsed.
\heading{Variable gadgets.}
For every variable $u$ we create a gadget $\var(u)$. This gadget has
three special edges; two are associated, respectively, with \true{} and
\false{}; we call the third one ``unlocking''. Overall, the
construction ensures that any removal of $\tilde \chi(K_\phi)$
triangles from $K_\phi$ either frees exactly one of the edges
associated with \true{} or \false{} in every variable gadget, or makes
$K_\phi$ obviously non-collapsible. This relates the removal of
triangles in $K_\phi$ to the assignment of variables in $\phi$. We
also ensure that part of each variable gadget remains uncollapsible
until the special unlocking edge is freed.
\heading{Clause gadgets.}
For every clause $c = \ell_1 \vee \ell_2 \vee \ell_3$ we create a
gadget $\cla(c)$. This gadget has three special edges, one per
literal $\ell_i$. Assume that $\ell_i\in \{u,\neg u\}$. Then the
special edge associated with $\ell_i$ is attached to $\var(u)$ so that
it can be freed if and only if the triangle removal phase freed the
special edge of $\var(u)$ associated with \true{} (if $\ell_i = u$) or
with \false{} (if $\ell_i = \neg u$). This ensures that the gadget
$\cla(c)$ can be collapsed if and only if one of its literals was
``selected'' at the triangle removal phase.
\heading{Conjunction gadget.}
We add a gadget $\gand$ with a single special edge, that is attached
to every clause gadget. This gadget can be collapsed only after the
collapse of every clause gadget has started (hence, if every clause
contains a literal selected at the triangle removal phase). In turn,
the collapse of $\gand$ will free the unlocking special edge of every
variable gadget, allowing to complete the collapse.
\heading{Notations.}
For any variable $u$, we denote the special edges of $\var(u)$
associated with \true{} and \false{} by, respectively, $f[u]$ and $f[\neg
  u]$; we denote the unlocking edge by $f(u)$. For every clause $c =
\ell_1 \vee \ell_2 \vee \ell_3$, we denote by $f[\ell_i,c]$ the
special edge of $\cla(c)$ associated with $\ell_i$. We denote by
$\fand$ the special edge of the conjunction gadget $\gand$. The
attachment of these edges are summarized in  
\ifsocg
\else
Table~\ref{t:summary}.
\fi
\ifsocg\else
\begin{table}[h]
\fi
\begin{center}
  \begin{tabular}{|c|c|c|c|}
  \hline
  gadget & special edges & attached to & freed by \\
  \hline
  \hline
  &  $f[u]$ & - & triangle deletion \\ [0.2ex]
  $\var(u)$ &  $f[\neg u]$ & - & triangle deletion \\ [0.2ex]
  & $f(u)$ & $\gand$ & freeing $\fand$ \\ [0.2ex]
  \hline
  & $f[u_2,c]$ & $\var(u_2)$ & freeing $f[u_2]$ \\ [0.2ex]
  $\cla(u_2 \vee \neg u_4 \vee u_9)$ & $f[\neg u_4,c]$ & $\var(u_4)$ & freeing $f[\neg u_4]$ \\ [0.2ex]
  & $f[u_9,c]$ & $\var(u_9)$ &freeing $f[u_9]$ \\ [0.2ex]
  \hline
  $\gand$ & $\fand$ & every clause gadget & collapsing all clause gadgets\\ [0.2ex]
  \hline
  \end{tabular}
\end{center}
\ifsocg\else
\caption{Summary of the gadgets' special edges and their attachments. \label{t:summary}}
\end{table}
\fi
\heading{Flow of collapses.}
Let us summarize the mechanism sketched above. Assume that
$\phi$ is satisfiable, and consider a satisfying assignment. Remove
the triangles from each $\var(u)$ so that the edge that becomes free
is $f[u]$ if $u$ was assigned \true{}, and $f[\neg u]$ otherwise.
This will allow to collapse each clause gadget in order to make
$\fand$ free.
Consequently, we will be able to 
collapse $\gand$ and make all unlocking edges $f(u)$ free. This allows
finishing the collapses on all $\var(u)$.

On the other hand, to collapse $K_\phi$ we must 
collapse $\fand$ at some point. Before this can happen, we have
to collapse in each clause $c = \ell_1 \vee \ell_2 \vee \ell_3$ one of
the edges $f[\ell_i,c]$. This, in turn, requires that $f[\ell_i]$ has
been made free. If we can ensure that $f[\neg \ell_i]$ cannot also be
free, then we can read off from the collapse an assignment of the
variables that must satisfy every clause, and therefore $\phi$.
(If $\ell_i = u$, then we set $u$ to
\true{}, if $\ell_i = \neg u$, then we set $u$ to \false{}. If there are unassigned
variables after considering all clause, we assign them arbitrarily.)
\subsection{Preparation: modified Bing's houses}
\label{s:Bing}

Our gadgets rely on two modifications of Bing's house, a classical
example of $2$-dimensional simplicial complex that is contractible but
not collapsible. Bing's house consists of a box split into two parts
(or \emph{rooms}); each room is connected to the outside by a tunnel
through the other room; 
each tunnel is attached to the room that it traverses by a rectangle
(or \emph{wall}). The modifications that we use here make the complex
collapsible, but restricts its set of free faces to exactly one or
exactly three edges.
\heading{One free edge.} 
We use here a modification due to Malgouyres and
Franc\'es~\cite{malgouyres-frances08}. In one of the rooms (say the
top one), the wall has been thickened and hollowed out, see
\ifsocg\else{} Figure~\ref{f:Bing1}.\fi{} We call the resulting polyhedron a Bing's house
with a single free edge, or a \emph{$1$-house} for
short. Two special elements of a $1$-house are its
\emph{free edge} (denoted $f$ and in thick stroke in
\ifsocg\else Figure~\ref{f:Bing1}) \fi and its \emph{lower wall} rectangle (denoted $L$
and colored in light blue in \ifsocg\else Figure~\ref{f:Bing1}). \fi We only consider
triangulations of $1$-house that subdivide the edge $f$
and the lower wall $L$. We use $1$-houses for the
following properties:

\ifsocg
\else
\begin{figure}[ht]
  \begin{center} \includegraphics[page=1]{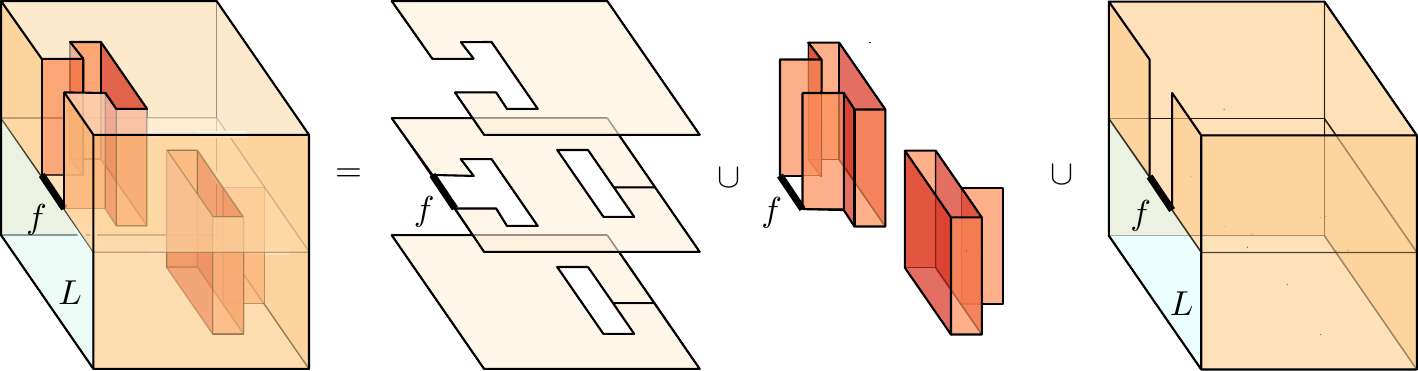}
  \ifMpic \includegraphics[page=1]{revised-figures_Martin} \fi
  \caption{Bing's house modified to be collapsible with exactly one
    free edge $f$.\label{f:Bing1}} \end{center}
\end{figure}
\fi

\begin{lemma}
  \label{l:bhsf}
  Let $B$ be a $1$-house, $f$ its free edge and $L$ its
  lower wall. In any triangulation of $B$, the free faces are exactly
  the edges that subdivide $f$. Moreover,~$B$ collapses to any subtree
  of the $1$-skeleton of $B$ that is contained in $L$ and shares with
  the boundary of~$L$ a single endpoint of $f$.
\end{lemma}

\noindent
The first statement follows from the fact that the edges that
  subdivide $f$ are the only ones that are not part of two triangles;
  see \cite[Remark~1]{malgouyres-frances08}. The second statement was
  proven in~\cite[Lemma 7]{tancer16} for certain trees, but the
  argument holds for arbitrary trees; we spell them out in
  Appendix~\ref{s:collapse_to_tree}.
  When working with $1$-houses, we will usually only describe the lower wall to clarify which
subtree we intend to collapse to.

\ifsocg
\else
\begin{remark}
 We note that there exist smaller 
 simplicial complexes that 
  have properties analogous to those of the $1$-house.
 The smallest such example is obtained by a slight modification of the dunce hat; has 
 seven vertices and thirteen facets and is described as the first example in~\cite[Section
 5.3]{hachimori2000}.\footnote{More precisely, the complex has vertices $\{1,2,\ldots, 7\}$, facets $\{125, 126, 127, 134, 145, 167, 234, 235, 236, 247, 356, 457, 567\}$, and its only free edge is $13$, contained in a unique facet $134$. A computer search confirmed that this (along with with four other triangulations of the modified dunce hat) is indeed a minimal example.}
For the purposes of exposition, however, we prefer to work with $1$-houses, which allows us to use some of their properties proved in~\cite{malgouyres-frances08,tancer16}. Moreover, the construction of the $1$-house is similar to the construction of another gadget, the $3$-house discussed below, and we currently do not know how to replace the latter by a smaller complex.
\end{remark}

\fi
\heading{Three free edges.}
We also use the Bing's houses with three collapsed walls introduced
in~\cite{tancer16}; we call them \emph{$3$-houses} for
short. These are $2$-dimensional complexes whose construction is more
involved; we thus state its main properties, so that we can use it as
a black box, and
\noindent
\ifsocg
\else
\ \\ \noindent
\fi
\ifsocg
\begin{minipage}{10cm}
\else
\begin{minipage}{11cm}
\fi
refer the  reader interested in its precise definition
to~\cite[$\mathsection 4$]{tancer16}.
Refer to the figure on the right
(which corresponds to Figure~9 in~\cite{tancer16}). 
The $3$-house has exactly three free edges $f_1, f_2, f_3$, and has
three distinguished paths $p_1, p_2, p_3$ sharing a common vertex $v$
and such that each $p_i$ shares exactly one vertex
with $f_i$ and no
vertex with $f_j$ for $j \neq i$. In addition, it contains an edge $e$
incident to $v$ so that the union of $p_1, p_2, p_3, f_1, f_2, f_3$
and $e$ forms a subdivided star with four rays.
\end{minipage}
\ifsocg
\else
\begin{minipage}{5cm}
  \begin{center} \includegraphics[keepaspectratio,page=2]{revised-figures} \end{center}
\end{minipage}
\fi

 Let $C$ denote the $3$-house as described
above. In~\cite{tancer16}, the polyhedron of $C$ is described in
detail but no triangulation is specified. We are happy with any
  concrete triangulation for which Lemma~\ref{l:clause} below holds;
  we can in addition require that the paths $p_1$, $p_2$ and $p_3$
  each consist of two edges.\footnote{The value two is not important here; what matters is to fix some value that can be used throughout the construction.}

\begin{lemma}[{\cite[Lemma~8]{tancer16}}]
  \label{l:clause}
  In any subdivision of $C$, the free faces are exactly the edges that
  subdivide $f_1$, $f_2$ and $f_3$. Moreover, $C$ collapses to the
  $1$-complex spanned by $e, p_1, p_2, p_3$ and any two of
  $\{f_1,f_2,f_3\}$.
\end{lemma}

\subsection{Detailed construction}
\label{s:detailed}

Section~\ref{s:outline} gave a quick description of the intended
functions of the various gadgets. We now flesh them out and describe
how they are glued together.

\heading{Triangulations.}
For some parts of the complex, it will be convenient to first describe
the polyhedron, then discuss its triangulation. Our description of the
triangulation may vary in precision: it may be omitted (if any
reasonable triangulation works), given indirectly by the properties it
should satisfy, or given explicitly (for instance to make it clear
that we can glue the gadgets as announced).

\ifsocg
\else
\bigskip
\fi
\noindent
\ifsocg
\else
\begin{minipage}{10cm}
\heading{Conjunction gadget.}
The conjunction gadget $\gand$ is a $1$-house. We let
$\fand$ denote its free edge and $\vand$ one of the endpoints of
$\fand$. We further triangulate the lower wall so that $\vand$ has
sufficiently high degree, allowing to assign every variable $u$ to an
internal edge $f(u)$ of the lower wall incident to $\vand$. 
See the lower left wall on the right picture.
  Any triangulation satisfying these prescriptions and computable in time
  polynomial in the size of $\phi$ suits our purpose. 
\end{minipage}
\fi
\ifsocg
\else
\hfill
\begin{minipage}{5cm}
  \begin{center} \includegraphics[page=6, keepaspectratio=true]{revised-figures} \end{center}
\end{minipage}
\fi

\ifsocg
\else
\bigskip
\fi
\heading{Variable gadget.}
The variable gadget $\var(u)$ associated with the variable $u$
has four parts.
\ifsocg

\noindent\else
\begin{itemize}
\item[]
\fi
    1. The first part is a triangulated $2$-sphere $S(u)$ that
    consists of two disks $D[u]$ and $D[\neg u]$ sharing a common
    boundary circle $s(u)$. The circle $s(u)$ contains a distinguished
    vertex $v(u)$. The disk $D[u]$ (resp. $D[\neg u]$) has a
    distinguished edge $f[u]$ (resp. $f[\neg u]$) that joins $v(u)$ to
    its center.
\ifsocg
\else
    \begin{center} \includegraphics[page=4, keepaspectratio=true]{revised-figures} \end{center}
\fi

\ifsocg\noindent\begin{minipage}{10cm}\else\item[]\begin{minipage}{10cm}\fi
  2. The second part is a $2$-complex $O(u)$ that consists of two
  ``boundary'' circles sharing a vertex. The vertex is identified with
  the vertex $v(u)$ of $S(u)$. One of the circles is identified with
  $s(u)$. The other circle is decomposed into two arcs: one is a
  single edge named $b(u)$, the other is a path with two edges which
  we call $p(u)$. The vertex common to $b(u)$ and $p(u)$, distinct from $v(u)$, is identified
  with the vertex~$\vand$ of the conjunction gadget.
\end{minipage}
\hfill
\ifsocg
\else
\begin{minipage}{5cm}
\begin{center} \includegraphics[keepaspectratio=true,page=5]{revised-figures} \end{center}
\end{minipage}
\fi

\ifsocg\smallskip\noindent\begin{minipage}{10cm}\else\item[]\begin{minipage}{10cm}\fi
  3. The third part is a $1$-house $B(u)$ intended to
  block the edge $b(u) \in O(u)$ from being free as long as the
  conjunction gadget has not been collapsed. The free edge of $B(u)$
  is identified with the edge $f(u)$ in the conjunction gadget $\gand$
  and the edge $b(u) \in O(u)$ is identified with an edge of the lower
  wall of $B(u)$ that shares the vertex $\vand$ with $f(u)$.
\end{minipage}
\hfill
\ifsocg
\else
\begin{minipage}{4cm}
  \begin{flushright} \includegraphics[keepaspectratio=true,page=7]{revised-figures} \end{flushright}
\end{minipage}
\fi

\ifsocg\smallskip\noindent\begin{minipage}{10cm}\else\item[]\begin{minipage}{10cm}\fi
  4. The fourth part consists of two complexes, $X[u]$ and $X[\neg
    u]$. Let $\ell \in \{u, \neg u\}$ and refer to the figure on the
  right. The complex $X[\ell]$ is a $1$-house whose free
  edge is identified with the edge $f[\ell]$ from $D[\ell]$, and whose
  lower wall contains a path identified with $p(u)$. Hence, $p(u)$ is
  common to $X[u]$, $X[\neg u]$ and the second part $O(u)$. For every
  clause $c_i$ containing the literal~$\ell$, we add in the lower wall
  a two-edge path $p[\ell,c_i]$ extended by an edge $f[\ell, c_i]$;
  the path $p[\ell,c_i]$ intersects $p(u)$ in exactly $\vand$ (in
  particular, these paths and edges form a subdivided star centered at
  $\vand$).
\end{minipage}
\hfill
\ifsocg
\else
\begin{minipage}{4cm}
  \begin{flushright} \includegraphics[keepaspectratio=true,page=8]{revised-figures} \end{flushright}
\end{minipage}  
\fi
\ifsocg\else\end{itemize}\fi

\smallskip
\noindent
\ifsocg
\begin{minipage}{10cm}
\else
\begin{minipage}{11cm}
\fi
\heading{Clause gadget.}
The clause gadget $\cla(c)$ associated with the clause $c = \ell_1
\vee \ell_2 \vee \ell_3$ is a $3$-house where:
\begin{itemize}
  \item the edges $f_i$ of $C$ are identified with the edges $f[\ell_i, c]$ in
    $X[\ell_i]$;
  \item the paths $p_i$ of $C$ are identified with the paths $p[\ell_i, c]$ in
    $X[\ell_i]$;
\item the vertex $v$ of $C$ is identified with the vertex $\vand$; and
\item the edge $e$ of $C$ is identified with the edge $\fand$.
\end{itemize}
\end{minipage}
\hfill
\ifsocg
\else
\begin{minipage}{4cm}
  \begin{center} \includegraphics[keepaspectratio, page=3]{revised-figures} \end{center}
\end{minipage}
\fi

\heading{Putting it all together.}
Let $\phi$ be a 3-CNF formula with variables $u_1,u_2, \ldots, u_n$
and clauses $c_1, c_2, \ldots, c_m$. The complex $K_\phi$ is defined as
\[ K_\phi = \gand \cup \pth{\bigcup_{i=1}^n \underbrace{S(u_i) \cup O(u_i) \cup  B(u_i) \cup X[u_i] \cup X[\neg u_i]}_{\var(u_i)}} \cup \pth{\bigcup_{j=1}^m \cla(c_j)}.\]

\ifsocg\else
To verify the proofs in Sections~\ref{s:links} to~\ref{s:reverse}, it
may be useful to be able to quickly identify for a given vertex, edge
or path which $2$-complexes contain it. We provide this in
Table~\ref{t:inclusions}.

\begin{table}
\begin{center}
\begin{tabular}{|c|c|c|c|}
\hline
  object & quantifier & in complexes \\ 
\hline\hline 
  $\vand$ & 1 occurence & $\gand$, $O(u)$, $B(u)$, $X[\ell]$, $\cla(c)$ \\ [0.2ex]
  $v(u)$ & every variable $u$ & $D[u]$, $D[\neg u]$, $O(u)$, $B(u)$,
$X[u]$, $X[\neg u]$ \\ [0.2ex]
\hline
$\fand$ & 1 occurence & $\gand$, $\cla(c)$ \\ [0.2ex]
  $f(u)$ & every variable $u$ & $\gand$, $B(u)$ \\ [0.2ex]
$f[\ell]$ & every literal $\ell$ & $D[\ell]$, $X[\ell]$ \\ [0.2ex]
  $b(u)$ & every variable $u$ & $O(u)$, $B(u)$ \\ [0.2ex]
  $p(u)$ & every variable $u$ & $O(u)$, $X[u]$, $X[\neg u]$ \\ [0.2ex]
  $s(u)$ & every variable $u$ & $O(u)$, $D[u]$, $D[\neg u]$ \\ [0.2ex]
  $f[\ell,c]$ & pairs $(\ell, c)$, $\ell \in c$ & $X[\ell]$, $\cla(c)$ \\ [0.2ex]
  $p[\ell,c]$ & pairs $(\ell, c)$, $\ell \in c$ & $X[\ell]$, $\cla(c)$ \\ [0.2ex]
  \hline
\end{tabular}
\caption{Containments of vertices, edges and paths in $2$-complexes.}
\label{t:inclusions}
\end{center}
\end{table}
\fi

\ifsocg
\else
\section{Connectedness of links}
\label{s:links}

In this section, we prove Proposition~\ref{p:Kphi}(i), \emph{i.e.}
that the link of every vertex in the complex $K_\phi$ of
Section~\ref{s:construction} is connected. By construction, the
complex $K_\phi$ is covered by the following subcomplexes: $\gand$,
$S(u) \cup O(u)$, $B(u)$, $X[\ell]$ and $\cla(c)$, where $u$ ranges
over all variables, $\ell$ ranges over all literals and $c$ ranges
over all clauses. We first argue that in each subcomplex, the link of
every vertex is connected. We then ensure that these subcomplexes are
glued into $K_\phi$ in a way that preserves the connectedness of the links.

\bigskip
\noindent
\begin{minipage}{5.5cm}
  \begin{center} \includegraphics[keepaspectratio, page=9]{revised-figures} \end{center}
\end{minipage}
\hfill
\begin{minipage}{9.5cm}
\heading{Links within the subcomplexes.}
The proof is straightforward but rather pedestrian. Let us start
  with the $1$-house $B$. Most points have a link homeomorphic to
  $K_3$, so connectedness is immediate. For
  the remaining points, depicted on the left, the links are
  homeomorphic to one of $K_2$, $K_4$, $K_{2,3}$ and $K_2 \cdot K_3$,
  the graph obtained by gluing $K_2$ and $K_3$ at a vertex.  This
settles the cases of the subcomplexes $\gand$, $B(u)$ and $X[\ell]$.
\end{minipage}

\bigskip
\noindent
\begin{minipage}{9.5cm}
  A similar analysis for the subcomplex $S(u) \cup O(u)$ reveals that
  the links are homeomorphic to one of $K_2$, $K_3$, $K_{2,3}$ and the
  `bull graph', that is the graph on 5 vertices formed from the
  triangle and two edges attached to it at different vertices. 
\end{minipage}
\hfill
\begin{minipage}{5.5cm}
  \begin{center} \includegraphics[keepaspectratio, page=10]{revised-figures} \end{center}
\end{minipage}

\bigskip

Last, a careful inspection of the construction in~\cite{tancer16} of
the $3$-house yields that the link of every vertex is
homeomorphic to one of $K_2$, $K_3$, $K_{2,3}$, $K_2 \cdot K_3$,
$K_{2,4}$ or the graph obtained from a triangle by gluing three other
triangles to it, one along each edge. This covers the subcomplexes
$\cla(c)$.

\heading{Links after gluing.}
We now argue that if $v$ is a vertex shared by two of our subcomplexes
$C$ and $C'$, then there is an edge incident to $v$ and common to $C$
and $C'$. This ensures that the links of $v$ in $C$ and $C'$ share at
least a vertex, so the connectedness of $\lk_{C \cup C'} v$ follows
from that of $\lk_{C} v$ and $\lk_{C'} v$. If $v$ is shared by
subcomplexes $C_1, C_2, \ldots C_k$, we can apply this idea
iteratively by finding a sequence of edges $e_1, e_2, \ldots, e_{k-1}$
where $e_i$ is common to $C_i$ and \emph{at least one of} $C_1, C_2,
\ldots, C_{i-1}$.

\bigskip

Let us first examine $v(u)$ for some variable $u$. This vertex is
common to $S(u) \cup O(u)$, $B(u)$, $X[u]$ and $X[\neg u]$. The
connectedness of $\lk_{K_\phi} v(u)$ follows from the existence of the
following edges incident to $v(u)$:
\begin{itemize}
\item $b(u)$, common to $S(u) \cup O(u)$ and $B(u)$,
\item $p(u)$, common to $S(u) \cup O(u)$, $X[u]$ and $X[\neg u]$.
\end{itemize}

\bigskip

Let us next examine $\vand$. This vertex is common to all our
subcomplexes, that is to $\gand$, $B(u)$, $S(u) \cup O(u)$, $X[\ell]$
and $\cla(c)$ for all variable $u$, literal $\ell$, and clause
$c$. The connectedness of $\lk_{K_\phi} \vand$ follows from the
existence of the following edges incident to $\vand$:
\begin{itemize}
\item the edges $f(u)$, common to $\gand$ and $B(u)$,
\item the edges $b(u)$, common to $B(u)$ and $S(u) \cup O(u)$,
\item the edges $p(u)$, common to $S(u) \cup O(u)$, $X[u]$ and $X[\neg u]$
\item $\fand$, common to $\gand$ and every $\cla(c)$.
\end{itemize}

\bigskip

The remaining vertices shared by two or more of our subcomplexes are
defined as part of an edge common to these subcomplexes. The
connectedness of the links in $K_\phi$ of these vertices is thus
immediate. This completes the proof of Proposition~\ref{p:Kphi}(i).

\fi

\ifsocg
\else
\section{Reduced Euler characteristic of $K_\phi$}
\label{s:euler}

In this section, we compute the reduced Euler characteristic of
$K_\phi$, preparing the proofs of Proposition~\ref{p:Kphi}(ii)--(iii)
in the following sections. By inclusion-exclusion, for any simplicial complexes
$K_1$ and $K_2$ we have:
\begin{equation}
  \label{e:chi}
  \tilde \chi (K_1 \cup K_2) = \tilde \chi
  (K_1) + \tilde \chi(K_2) - \tilde \chi (K_1 \cap K_2).
\end{equation}
In particular, if both $K_2$ and $K_1 \cap K_2$ are contractible, then
$\tilde \chi(K_1 \cup K_2) = \tilde \chi (K_1)$.

\begin{proposition}
\label{p:chi}
 $\tilde \chi (K_\phi)$ equals the number of variables of $\phi$.
\end{proposition}

\begin{proof}
First, let us observe that the subcomplexes $\gand$, $B[u]$, $X[\ell]$ and
  $\cla(c)$ (for all variables $u$, literals $\ell$ and clauses $c$) are
  contractible. Indeed each of them is either a $1$-house or $3$-house which
  are collapsible by Lemmas~\ref{l:bhsf} and~\ref{l:clause}, thereby
  contractible. In addition, each of the aforementioned subcomplexes is attached to the rest of the complex in contractible subcomplexes (trees).

  Therefore, by the claim following Equation~\eqref{e:chi}, 
  we may replace each of these gadgets (in
  any order) with the shared trees without affecting the reduced Euler
  characteristic. That is, $\tilde \chi (K_\phi) = \tilde \chi (K')$
  where
$$
  K' := \fand \cup \bigcup_{u} (f(u) \cup S(u) \cup O(u)) \cup \bigcup_{(\ell, c):  \ell \in c} (f[\ell,c] \cup p[\ell,c])
$$
  where the first (big) union is over all variables $u$, and the second is over all pairs $(\ell, c)$ where a literal
  $\ell$ belongs to a clause $c$.

\smallskip\noindent
\begin{minipage}{10cm}

\qquad By collapsing the pendent edges and paths, we get $\tilde \chi (K') = \tilde
  \chi(K'')$ where
$$
K'' := \bigcup_{u} \left(S(u) \cup O(u)\right).
$$

\qquad Finally, for every variable $u$ we have
  $\tilde \chi (O(u) \cup S(u)) = 1$ as $O(u) \cup S(u)$ is homotopy
  equivalent to the $2$-sphere. For any distinct variables $u,u'$, the
  complexes $O(u) \cup S(u)$ and $O(u') \cup S(u')$ share only a
  vertex, namely~$\vand$. Equation~\eqref{e:chi} then yields that
  $\tilde \chi(K_\phi) = \tilde \chi(K'')$ is the number of variables.
\end{minipage}
\hfill
\begin{minipage}{5cm}
  \begin{center} \includegraphics[page=11]{revised-figures} \end{center}
\end{minipage}

\end{proof}

\begin{remark}
  It is possible, with slightly more effort, to show that $K_\phi$ is
  homotopy equivalent to 
  $K''$, hence 
  to a 
  wedge of spheres, one for each variable.
  This also implies Proposition~\ref{p:chi} but for our
  purpose, computing the reduced Euler characteristic suffices.
\end{remark}

\fi

\section{Satisfiability implies collapsibility}
\label{s:direct}
In this section we prove Proposition~\ref{p:Kphi}(ii), \emph{i.e.}
that if $\phi$ is satisfiable, then there exists a choice of
$\tilde{\chi}(K_\phi)$ triangles of $K_\phi$ whose removal makes the
complex collapsible.
\ifsocg
\else
\heading{Literal $\ell(u)$.}
Let us fix a satisfying assignment for $\phi$. For every variable $u$,
we set $\ell(u)$ to $u$ if $u$ is \true{}~in our assignment, and to
$\neg u$ otherwise.

\heading{Triangle removal.}
For every variable $u$, we remove a triangle from the region
$D[\ell(u)]$ of the sphere~$S(u)$. Proposition~\ref{p:chi} ensures
that this removes precisely $\tilde{\chi}(K_\phi)$ triangles, as
announced.
\fi
\heading{Constrain complex.}
It will be convenient to analyze collapses of $K_\phi$ locally within
a subcomplex, typically a gadget. To do so formally, we use constrain
complexes following~\cite{tancer16}. Given a simplicial complex $K$
and a subcomplex $M$ of $K$, we define the \emph{constrain complex} of
$(K,M)$, denoted $\Gamma(K,M)$, as 
\ifsocg
\else
 follows:
\[  \Gamma(K,M) := \{\vartheta  \in M\colon  \exists \eta \in K \setminus M \hbox{ s.t. } \vartheta \subset \eta\}.\]
\fi

\begin{lemma}[{\cite[Lemma~4]{tancer16}}]
  \label{l:constrain}
  Let $K$ be a simplicial complex and $M$ a subcomplex of $K$. If $M$
  collapses to $M'$ and $\Gamma(K,M) \subseteq M'$ then $K$ collapses
  to $(K \setminus M) \cup M'$.
\end{lemma}
\heading{Collapses.}
We now describe a sequence of collapses enabled by the removal of the
triangles. Recall that we started from the complex
\ifsocg
\else
\[ K_\phi = \gand \cup \pth{\bigcup_{i=1}^n O(u_i) \cup S(u_i) \cup B(u_i) \cup X[u_i] \cup X[\neg u_i]} \cup \pth{\bigcup_{j=1}^m \cla(c_j)}\]
\fi
  where $u_1,u_2, \ldots, u_n$ and $c_1, c_2, \ldots, c_m$ are, respectively, the variables and the clauses of $\phi$. We then removed  a triangle from each $D[\ell(u)]$.
\ifsocg
\else
\begin{enumerate}[(a)]
\item \fi
  The removal of a triangle of $D[\ell(u)]$ allows to collapse
  that subcomplex to $s(u) \cup f[\ell(u)]$. This frees
  $f[\ell(u)]$. The complex becomes:
\ifsocg
\else
  \[ K_a = \gand \cup \pth{\bigcup_{i=1}^n O(u_i) \cup D[\neg \ell(u_i)] \cup B(u_i) \cup X[u_i] \cup X[\neg u_i]} \cup \pth{\bigcup_{j=1}^m \cla(c_j)}.\]
\fi
\ifsocg

\else
\item \fi 
  We can then start to collapse the subcomplexes $X[\ell(u)]$. We
  proceed one variable at a time. Assume that we are about to proceed
  with the collapse of $X[\ell(u)]$ and let $K$ denote the current
  complex. Locally, $X[\ell(u)]$ is a $1$-house with free edge $f[\ell(u)]$. Moreover, $\Gamma(K,X[\ell(u)])$ is the tree $T(u)$
  formed by the path $p(u)$ and the union of the paths $p[\ell(u),c]
  \cup f[\ell(u),c]$ for every clause $c$ using the literal
  $\ell(u)$. Lemma~\ref{l:bhsf} ensures that $X[\ell(u)]$ can be
  locally collapsed to $T(u)$, and Lemma~\ref{l:constrain} ensures that
  $K$ can be globally collapsed to $(K \setminus X[\ell(u)]) \cup
  T(u)$. We proceed in this way for every complex $X[\ell(u)]$. The
  complex becomes:
\ifsocg
\else
\[ K_b = \gand \cup \pth{\bigcup_{i=1}^n O(u_i) \cup D[\neg \ell(u_i)] \cup B(u_i) \cup X[\neg \ell(u_i)]} \cup \pth{\bigcup_{j=1}^m \cla(c_j)}.\]
\fi
\ifsocg

\else
\item \fi The collapses so far have freed every edge of $f[\ell(u_i),c]$. We
  now consider every clause $c_j$ in turn. Put $c_j = (\ell_1 \vee
  \ell_2 \vee \ell_3)$ and let $K$ denote the current complex. The
  assignment that we chose is satisfying, so at least one of $\ell_1$,
  $\ell_2$ or $\ell_3$ coincides with $\ell(u_i)$ for some $i$; let us
  assume without loss of generality that $\ell_1 = \ell(u_i)$. The
  edge $f[\ell_1,c]$ is therefore free and Lemma~\ref{l:clause} yields
  that locally, $\cla(c_j)$ collapses to the tree $T(c_j) = \fand \cup
  p[\ell_1,c] \cup p[\ell_2,c] \cup p[\ell_2,c] \cup f[\ell_2,c] \cup
  f[\ell_3,c]$. Moreover, $\Gamma(K,\cla(c_j)) = T(c_j)$ so
  Lemma~\ref{l:constrain} ensure that $K$ can be globally collapsed to
  $(K \setminus \cla(c_j)) \cup T(c_j)$. After proceeding in this way for
    every complex $\cla(c_j)$, \ifsocg\else the complex becomes:\fi
\ifsocg
\else
\[ K_c = \gand \cup \pth{\bigcup_{i=1}^n O(u_i) \cup D[\neg \ell(u_i)] \cup B(u_i) \cup X[\neg \ell(u_i)]} \cup \pth{\bigcup_{j=1}^m T(c_j)}.\]
\fi

\ifsocg

\else
\item \fi The collapses so far have freed the edge $\fand$. We can then
  proceed to collapse $\gand$. Locally, Lemma~\ref{l:bhsf} allows to
  collapse $\gand$ to the tree $T = f(u_1) \cup f(u_2) \cup \ldots
  \cup f(u_n)$. (From this point, we expect the reader to be able to
  check by her/himself that Lemma~\ref{l:constrain} allows to perform
  globally the collapse described locally.) The complex becomes:
\ifsocg
\else
\[ K_d = \pth{\bigcup_{i=1}^n O(u_i) \cup D[\neg \ell(u_i)] \cup B(u_i) \cup X[\neg \ell(u_i)]} \cup \pth{\bigcup_{j=1}^m T(c_j)}.\]
\fi

\ifsocg

\else
\item \fi The collapses so far have freed every edge $f(u_i)$. Thus,
  Lemma~\ref{l:bhsf} allows to collapse each complex $B(u_i)$ to its
  edge $b(u_i)$. This frees the edge $b(u_i)$, so the complex $O(u_i)$
  can in turn be collapsed to $s(u_i) \cup p(u_i)$. At this point, the
  complex is:
\ifsocg
\else
\[ K_e = \pth{\bigcup_{i=1}^n s(u_i) \cup p(u_i) \cup D[\neg \ell(u_i)] \cup X[\neg \ell(u_i)]} \cup \pth{\bigcup_{j=1}^m T(c_j)}.\]
\fi
\ifsocg

\else
\item \fi The collapses so far have freed every edge $s(u_i)$. We can thus
  collapse each $D[\neg \ell(u_i)]$ to $f[\neg \ell(u_i)]$. This frees
  every edge $f[\neg \ell(u_i)]$, allowing to collapse every
  subcomplex $X[\neg \ell(u_i)]$, again by Lemma~\ref{l:bhsf}, to the
  tree formed by the path $p(u_i)$ and the union of the paths $p[\neg
    \ell(u_i),c] \cup f[\neg \ell(u_i),c]$ for every clause $c$ using
  the literal $\neg \ell(u_i)$.
\ifsocg\else\end{enumerate}\fi

\ifsocg\else\noindent\fi
At this point, we are left with a $1$-dimensional complex. This
complex is a tree (more precisely a subdivided star centered in
$\vand$ and consisting of the paths $p(u_i)$, the paths $p[\ell,c]$
and some of the edges $f[\ell,c]$). As any tree is collapsible, this
completes the proof of Proposition~\ref{p:Kphi}(ii).

\section{Collapsibility implies satisfiability}
\label{s:reverse}

In this section we prove Proposition~\ref{p:Kphi}(iii), \emph{i.e.} we
consider some arbitrary subdivision $K'_\phi$ of $K_\phi$, and prove
that if $K'_\phi$ becomes collapsible after removing some $\tilde \chi
(K_\phi)$ triangles, then $\phi$ is satisfiable. We thus consider a
collapsible subcomplex $\wh K$ of $K'_\phi$ obtained by removing
$\tilde \chi(K_\phi)$ triangles from $K'_\phi$. 

\heading{Notations.}
Throughout this section, we use the following conventions. In
  general, we use hats (for example $\wh K$) to denote subcomplexes of
  $K'_\phi$. Given a subcomplex $M$ of $K_\phi$, we also write $M'$
  for the subcomplex of $K'_\phi$ that subdivides $M$.

\heading{Variable assignment from triangle removal.}
We first read our candidate assignment from the triangle removal
following the same idea as in Section~\ref{s:direct}. This relies on
two observations:
\begin{itemize}
\item The set of triangles removed in $\wh K$ contains exactly one
  triangle from each sphere $S'(u)$. Indeed, since $\wh K$ is
  collapsible and $2$-dimensional, it cannot contain a $2$-dimensional
  sphere. Hence, every sphere $S'(u)$ had at least one of its
  triangles removed. By Proposition~\ref{p:chi}, $\chi(K_\phi) =
  \chi(K'_\phi)$ equals the number of variables of $\phi$, so this
  accounts for all removed triangles.
\item For any variable $u$, any removed triangle in $S'(u)$ is either in
  $D'[u]$ or in $D'[\neg u]$. We give $u$ the \true{}~assignment in the
  former case and the \false{} assignment in the latter case.
\end{itemize}
The remainder of this section is devoted to prove that this assignment
satisfies $\phi$. It will again be convenient to denote by $\ell(u)$
the literal corresponding to this assignment, that is, $\ell(u) = u$
if $u$ was assigned \true{}~and $\ell(u) = \neg u$ otherwise.

\heading{Analyzing the collapse.}
Let us fix some collapse of $\wh K$. We argue that our assignment
satisfies $\phi$ by showing that these collapses must essentially
follow the logical order of the collapse constructed in
Section~\ref{s:direct}. To analyze the dependencies in the collapse,
it is convenient to consider the partial order that it induces on the
simplices of $\wh K$: $\sigma \prec \tau$ if and only if in our
collapse, $\sigma$ is deleted before $\tau$.  We also write $\sigma
\prec \wh M$ for a subcomplex $\wh M$ of $\wh K$ if $\sigma$ was
removed before removing any simplex of $\wh M$.

\bigskip

The key observation is the following dependency:

\begin{lemma}\label{l:egand}
  There exists an edge $\wh e$ of $\gand'$ such that $\wh e \prec
  D'[\neg \ell(u)]$ for every variable $u$.
\end{lemma}
\begin{proof}
  We first argue that for every variable $u$, there exists an edge
  $\wh e_1(u) \in b'(u) \cup p'(u)$ such that $\wh e_1(u) \prec
  D'[\neg \ell(u)]$. To see this, remark that the complex $D'[\neg
    \ell(u)] \cup O'(u)$ is fully contained in $\wh K$ since the
    triangle removed from $S'(u)$ belongs to $D'[\ell(u)]$. It thus
  has to be collapsed. Since this complex is a disk, the first
  elementary collapse in $D'[\neg \ell(u)] \cup O'(u)$ has to involve
  some edge $\wh e_1(u)$ of its boundary. This boundary is $b'(u) \cup
  p'(u)$, so it contains no edge of $D'[\neg \ell(u)]$. It follows
  that $\wh e_1(u) \prec D'[\neg \ell(u)]$.

  We next claim that $\wh e_1(u) \in b'(u)$. Indeed, remark that every
  edge in $p'(u)$ belongs to two triangles of $X'[\neg \ell(u)]$. By
  Lemma~\ref{l:bhsf}, any collapse of $X'[\neg \ell(u)]$ must start by
  an elementary collapse using an edge of $f'[\neg \ell(u)]$ as a free
  face. Any edge of $f'[\neg \ell(u)]$ is, however, contained in two
  triangles of $D'[\neg \ell(u)]$ and thus cannot precede $D'[\neg
    \ell(u)]$ in $\prec$. It follows that $\wh e_1(u) \in b'(u)$.
  
  We can now identify $\wh e$. Observe that $b'(u) \subset B'(u)$. As
  $B'(u)$ is a $1$-house, Lemma~\ref{l:bhsf} ensures that
  the first edge removed from $B'(u)$ must subdivide $f'(u)$. Hence, there is
  an edge $\wh e_2(u) \subset f'(u)$ such that $\wh e_2(u) \prec \wh
  e_1(u)$. Since $f'(u) \subset \gand'$, another $1$-house, the same reasoning yields an edge $\wh e_3(u)$ in $\fand'$
  such that $\wh e_3(u) \prec \wh e_2(u)$. Let $\wh e$ denote the
  first edge removed from $\gand'$ among all edges $\wh e_3(u)$. At
  this point, we have for every variable $u$
$  \wh e \prec \wh e_2(u)\prec \wh e_1(u) \prec D'[\neg \ell(u)]$, as announced.
\end{proof}

\noindent
Let $\wh e$ denote the edge of $\gand'$ provided by
Lemma~\ref{l:egand}, \emph{i.e.} satisfying $\wh e \prec D'[\neg
  \ell(u)]$ for every variable $u$. We can now check that the variable
assignment does satisfy the formula:
\begin{itemize}
\item Consider a clause $c = (\ell_1 \vee \ell_2 \vee \ell_3)$ in
  $\phi$. The complex $\cla'(c)$ is a $3$-house, so
  Lemma~\ref{l:clause} restricts its set of free edges to the
  $f'[\ell_i,c]$. Hence, there is $i \in \{1,2,3\}$ and an edge $\wh
  e_4(c)$ in $f'[\ell_i,c]$ such that $\wh e_4(c) \preceq \cla'(c)$.
  Note that, in particular, $\wh e_4(c) \prec \wh e$ as the edge
  $\fand$ also belongs to $\cla(c)$ and must be freed before
  collapsing $\gand'$ (by Lemma~\ref{l:bhsf}).

\item The subcomplex $f'[\ell_i,c]$ is contained not only in
  $\cla'(c)$, but also in $X[\ell_i]$ which is a $1$-house
  with free edge $f[\ell_i]$. By Lemma~\ref{l:bhsf}, the first
  elementary collapse of $X[\ell_i]$ uses as free face an edge $\wh
  e_5(c)$ that subdivides $f'[\ell_i]$.  In particular, $\wh e_5(c)
  \prec f'[\ell_i,c]$ and $\wh e_5(c) \prec \wh e_4(c)$.
  
\item Let $u$ be the variable of the literal $\ell_i$, that is, $\ell_i =
  u$, or $\ell_i = \neg u$; in particular
  $\ell_i \in \{\ell(u), \neg \ell(u)\}$. From $\wh e_5(c) \prec \wh
  e_4(c) \prec \wh e \prec D'[\neg \ell(u)]$ it comes that $\wh
  e_5(c)$ cannot belong to $D'[\neg \ell(u)]$. Yet, $\wh e_5(c)$
  belongs to $f'[\ell_i]$. It follows that $\ell_i \neq \neg \ell(u)$
  and we must have $\ell_i = \ell(u)$.
   The definition of $\ell(u)$ thus ensures that our assignment
   satisfies the clause $c$.        
\end{itemize}
Since our assignment satisfies every clause, it satisfies $\phi$.

\ifsocg
\else
\section*{Acknowledgment}
We would like to thank Russ Woodroofe and Andr\'{e}s D. Santamar\'{\i}a Galvis for
remarks on a preliminary version on this paper.
\fi

\ifsocg
\else
\bibliographystyle{alpha}
\fi
\bibliography{shard}

\clearpage
\appendix

\section{Collapsing a $1$-house}
\label{s:collapse_to_tree}

In this section we prove the second statement of
Lemma~\ref{l:bhsf} (recalled below). We use an auxiliary lemma:

\begin{lemma}\label{l:KtoT}
  A triangulation of a topological disk collapses to any tree
  contained in its $1$-skeleton.
\end{lemma}
\begin{proof}
  Let $K$ be a triangulation of a topological disk (that is, the polyhedron of
  $K$ is homeomorphic to a $2$-dimensional disk) and $T$ a tree contained in
  the $1$-skeleton of $K$. While there is an edge of $K$ that is free
  and not in $T$, we collapse such an edge. Let $K'$ denote the
  resulting complex. 

  Let us first argue that $K'$ contains no triangle. Let $c$ denote
  the (possibly empty) $\Z_2$-chain obtained by summing the triangles
  of $K'$. The $1$-chain $\partial c$ is a $1$-cycle by definition and
  it is supported on $T$. Indeed, every edge in $K'$ is contained in
  zero, one or two triangles and any edge contained in exactly one
  triangle and not in $T$ could be used as a free face to further
  collapse $K'$. Since a tree does not contain any nontrivial
  $1$-cycle, it follows that $\partial c$ is empty. In $K$, any
  nonempty $2$-chain has a non-empty boundary. It follows that $c$ is
  empty and $K'$ is indeed $1$-dimensional.

  Since $K'$ is a collapse of $K$, it must be contractible. Hence,
  $K'$ is a tree and, by construction, it contains $T$. The statement
  follows since a tree always collapses to any of its sub-trees.
\end{proof}

\noindent
We can now prove the announced statement: that a $1$-house
with free edge $f$ and lower wall $L$ collapses to any subtree $t$ of the
$1$-skeleton of $B$ that is contained in $L$ and shares with the
boundary of~$L$ a single endpoint of $f$.

\begin{proof}[Proof of Lemma~\ref{l:bhsf}]
  We apply Lemma~\ref{l:KtoT} repeatedly. First, we collapse the lower
  wall $L$ to the tree formed by $t$ and the subcomplex of $B$
  triangulating $(\partial L) \setminus f$. Next, we collapse the
  lowest floor, except for the edges that belong to walls that are
  still present. We proceed to collapse every wall that used to be
  attached to the lowest floor. The resulting complex is already a
  disk with $t$ attached to it. We collapse the disk to the attachment
  point and are done.
\end{proof}

\end{document}